\documentclass{amsart}

\usepackage[english]{babel}
\usepackage[ansinew]{inputenc}
\usepackage{amsmath, amsthm, amssymb, graphicx, mathrsfs}
\usepackage{color}
\usepackage[colorlinks,citecolor=blue,linkcolor=red]{hyperref}
\usepackage{pinlabel}

\usepackage{amsmath,amsthm,amssymb}

\newtheorem{theor}{Theorem}

\newtheorem{corol}[theor]{Corollary}

\newtheorem{thm}{Theorem}[section]
\newtheorem{lem}[thm]{Lemma}
\newtheorem{cor}[thm]{Corollary}

\newtheorem{prop}[thm]{Proposition}
\theoremstyle{definition}
\newtheorem{defn}[thm]{Definition}

\newtheorem{rem}[thm]{Remark}

\newtheorem{prop-defn}[thm]{Proposition-Definition}

\newtheorem{claim}{Claim}
\newtheorem*{claim*}{Claim}
\newtheorem*{ack*}{Acknowledgements}

\title{Generalized north-south dynamics on the space of geodesic currents}
\author{Caglar Uyanik}
\address{\tt Department of Mathematics, University of Illinois at
 Urbana-Champaign, 1409 West Green Street, Urbana, IL 61801, USA
\newline http://www.math.uiuc.edu/\~{}cuyanik2/} \email{\tt cuyanik2@illinois.edu}

\begin{document}
	
\begin{abstract} We prove uniform north-south dynamics type results for the action of $\varphi\in Out(F_{N})$ on the space of projectivized geodesic currents $\mathbb{P}Curr(S)=\mathbb{P}Curr(F_N)$, where $\varphi$ is induced by a pseudo-Anosov homeomorphism on a compact surface S with boundary such that $\pi_{1}(S)=F_{N}$. As an application, we show that for a subgroup $H\le Out(F_N)$, containing a fully irreducible (iwip) element, either $H$ contains a hyperbolic iwip or $H$ is contained in the image in $Out(F_N)$ of the mapping class group of a surface with a single boundary component.
\end{abstract}

\thanks{ The author is partially supported by the NSF grants of Ilya Kapovich (DMS-0904200) and Christopher J. Leininger (DMS-1207183)}


\maketitle

\tableofcontents
	
\section{Introduction}

A well known result of Thurston \cite{Th} states that, if $S_{g}$ is a hyperbolic surface and $f\in Mod(S)$ is a pseudo-Anosov homeomorphism, then $f$ acts on the space of \emph{projective measured laminations} $\mathbb{P}\mathcal{M}\mathcal{L}(S)$ with north-south dynamics. In other words, there are two fixed points of this action, $[\mu_{+}]$ and $[\mu_{-}]$ called stable and unstable laminations, and any point $[\mu]\in \mathbb{P}\mathcal{M}\mathcal{L}(S)$ other than $[\mu_{-}]$ and $[\mu_{+}]$ converges to $[\mu_{+}]$ under positive iterates of $f$, and converges to $[\mu_{-}]$ under negative iterates of $f$. Moreover, there is a constant $\lambda>1$, called the \emph{dilatation}, such that $f\mu_{+}=\lambda\mu_{+}$ and $f^{-1}\mu_{-}=\frac{1}{\lambda}\mu_{-}$.  In fact, this convergence is uniform on compact sets by work of Ivanov \cite{Iva}.

A measured lamination $\mu\in\mathcal{M}\mathcal{L}(S)$ has two alternative interpretations:
\begin{enumerate}
\item as a geodesic current on $S$; see \cite{Bo86,Bo88},
\item as dual to an $\mathbb{R}$-tree; see for example \cite{Kap}, 
\end{enumerate}
both of these are very useful in terms of understanding mapping class groups.  

Let $F_{N}$ be a free group of rank $N\ge2$. The outer automorphism group of $F_{N}$, $Out(F_{N})$, is the quotient group $Aut(F_{N})/Inn(F_N)$. The group $Out(F_{N})$ is a close relative to $Mod(S)$ in the sense that tools for analyzing mapping class groups have similar counterparts in the study of elements of $Out(F_{N})$. The first such tool is the \emph{ Culler-Vogtmann's Outer Space \cite{CV}}, $cv_N$, which is the space of marked metric graphs or equivalently the space of minimal, free, discrete isometric actions of $F_{N}$ on $\mathbb{R}$-trees. This space $cv_N$ and its projectivization $CV_N$, obtained as the quotient by the action of $\mathbb{R}_+$, acting by scaling the metrics, both have natural closures $\overline{cv}_N$ and $\overline{CV}_N$ with respect to \emph{Gromov-Hausdorff} topology. We explain this in detail in Section \ref{2.1}. 

Another such tool is the space of geodesic currents on $F_{N}$, denoted by $Curr(F_N)$, which is the space of locally finite Borel measures on $\partial^{2}F_{N}$ which are $F_{N}$ invariant and \emph{flip} invariant. The space of \emph{projectivized geodesic currents}, denoted by $\mathbb{P}Curr(F_{N})$ is the quotient of $Curr(F_{N})$, where two currents are equivalent if they are positive scalar multiples of each other. 

The main $Out(F_{N})$ analog of a pseudo-Anosov mapping class is a \emph{fully irreducible} element, which is also called an \emph{irreducible with irreducible powers (iwip)}. An iwip $\varphi\in Out(F_{N})$ is an element such that no positive power of $\varphi$ fixes the conjugacy class of a proper free factor of $F_{N}$. There are two types of iwips both of which are equally important. An iwip $\varphi\in Out(F_{N})$ is called \emph{hyperbolic} or \emph{atoroidal} if it has no nontrivial periodic conjugacy classes. It is known \cite{BF, Brink} that an iwip is hyperbolic if and only if $G=F_{N}\rtimes_\varphi\mathbb{Z}$ is word-hyperbolic. An iwip $\varphi$ is called \emph{geometric} or \emph{non-atoroidal} otherwise. The reason for the name is a  theorem of Bestvina-Handel \cite{BH92}, saying that $\varphi\in Out(F_N)$ is a non-atoroidal iwip if and only if $\varphi$ is induced by a pseudo-Anosov homeomorphism of a surface $S$ with one boundary component and $\pi_1(S)\cong F_N$.

Thurston's north south dynamics result on $\mathbb{P}\mathcal{M}\mathcal{L}(S)$ has several different generalizations in the $Out(F_{N})$ context. The first such generalization is due to Levitt and Lustig. In \cite{LL} they show that for an iwip $\varphi\in Out(F_{N})$, its action on the compactified outer space $\overline{CV}_N$ has exactly two fixed points, $[T_{+}]$ and $[T_{-}]$. For any other point $[T]\neq[T_{\pm}]$ we have $\lim_{n\to\infty}[T]\varphi^{n}=[T_{+}]$ and $\lim_{n\to\infty}[T]\varphi^{-n}=[T_{-}]$. Moreover, this convergence is uniform on compact subsets $K$ of $\overline{CV}_N$.
  
Reiner Martin, in his unpublished 1995 thesis \cite{Martin}, proves that if $\varphi\in Out(F_{N})$ is a hyperbolic iwip, then $\varphi$ acts on $\mathbb{P}Curr(F_{N})$ with north-south dynamics. We discuss his results after Corollary F below. 

The first main result of this paper deals with the dynamical properties of a pseudo-Anosov map $f\in Mod(S)$ on $\mathbb{P}Curr(S)$ where $S$ is a compact, hyperbolic surface with $b\ge1$ boundary components $\alpha_1,\alpha_2,\dotsc,\alpha_b$. Let $\mu_{\alpha_i}$ be the current corresponding to the boundary curve $\alpha_i$. Here we note that that $\pi_1(S)=F_N$ and $\mathbb{P}Curr(F_N)=\mathbb{P}Curr(S)$. See section \ref{2.1} for details.

Let us define $\Delta, H_{0}(f), H_{1}(f)\subset\mathbb{P}Curr(S)=\mathbb{P}Curr(F_N)$ as follows:
\[
\Delta:=\{[a_1\mu_{\alpha_1}+a_2\mu_{\alpha_2}+\dotsc+a_b\mu_{\alpha_b}]\ |\ a_i\ge0,  \sum_{i=1}^{b}a_i>0 \}.
\]
\[
H_{0}(f):=\{[t_1\mu_{-}+t_2\nu]\ | [\nu]\in\Delta, t_1, t_2\ge0\}\ and\ H_{1}(f):=\{[t'_1\mu_{+}+t'_2\nu]\ |\  [\nu]\in\Delta, t'_1, t'_2\ge0\}.
\]

\begin{theor}\label{A}
Let $f$ be a pseudo-Anosov homeomorphism on $S$. Let $K$ be a compact set in $\mathbb{P}Curr(S)\setminus H_{0}(f)$. Then,
for any open neighborhood $U$ of $[\mu_{+}]$, there exist $m\in\mathbb{N}$ such that $f^{n}(K)\subset U$ for all $n\geq m$. Similarly for a compact set $K'\subset\mathbb{P}Curr(S)\setminus H_{1}(f)$ and an open neighborhood $V$ of $[\mu_{-}]$, there exist $m'\in\mathbb{N}$ such that $f^{-n}(K')\subset V$ for all $n\ge m'$. 
\end{theor}

Theorem A implies that if $[\nu]\in\mathbb{P}Curr(S)\setminus (H_{0}(f)\cup H_{1}(f))$, then $\lim_{n\to\infty}f^{n}[\nu]=[\mu_{+}]$ and $\lim_{n\to\infty}f^{-n}[\nu]=[\mu_{-}]$. 
Moreover, it is not hard to see that $f$ has \emph{exceptional dynamics} on $H_0(f)\cup H_1(f)$:

\noindent If $[\mu]=[t_1\mu_{+}+t_2\nu]$ where $t_1>0, [\nu]\in\Delta$, then \[
\lim_{n\to\infty}f^n([\mu])=[\mu_{+}]
\]
but 
\[
\lim_{n\to\infty}f^{-n}([\mu])=[t_2\nu].
\]  
If $[\mu]=[t'_1\mu_{-}+t'_2\nu]$ where $t'_1>0, [\nu]\in\Delta$, then
\[
\lim_{n\to\infty}f^{-n}([\mu])=[\mu_{-}]
\]but\[
\lim_{n\to\infty}f^{n}([\mu])=[t_2\nu].
\]
As a particular case of Theorem \ref{A} for surfaces with one boundary component we obtain the following general result about dynamics of non-atoroidal iwips on $\mathbb{P}Curr(F_N)$. 

\begin{theor}\label{B} Let $\varphi\in Out(F_{N})$ be a non-atoroidal iwip. Then the action of $\varphi$ on the space of projectivized geodesic currents, $\mathbb{P}Curr(F_{N})$, has uniform dynamics in the following sense: Given an open neighborhood $U$ of the stable current $[\mu_{+}]$ and a compact set $K_0\subset\mathbb{P}Curr(F_{N})\setminus H_{0}(\varphi)$, there is an integer $M_0>0$ such that for all $n\ge M_0$, $\varphi^{n}(K_0)\subset U$.  Similarly, given an open neighborhood $V$ of the unstable current $[\mu_{-}]$ and a compact set $K_1\subset\mathbb{P}Curr(F_{N})\setminus H_{1}(\varphi)$, there is an integer $M_1>0$ such that for all $m\ge M_1$, $\varphi^{-m}(K_1)\subset V$.
\end{theor} 
As a corollary of Theorem \ref{B} we get a unique-ergodicity type statement for non-atoroidal iwips:
\begin{corol}
Let $\varphi\in Out(F_{N})$ be a non-atoroidal iwip. Let $T_{+}$ and $T_{-}$ be representatives of attracting and repelling trees in $\overline{cv}_N$ corresponding to right action of $\varphi$ on $\overline{cv}_N$ respectively. Then 
\[
\left<T_{+},\mu\right>=0 \iff [\mu]=[a_{0}\mu_{-}+b_{0}\mu_{\alpha}]
\] for some $a_{0}\ge0,b_{0}\ge0$, where $\mu_{\alpha}$ is the current corresponding to boundary curve $\alpha$ as in Theorem \ref{B}. Similarly,
\[
\left<T_{-},\mu\right>=0 \iff [\mu]=[a_{1}\mu_{+}+b_{1}\mu_{\alpha}]
\] for some $a_{1}\ge0, b_{1}\ge0$.
\end{corol}

As an application of our methods, we obtain the following general structural result, which is a \emph{subgroup version} of a theorem of Bestvina-Handel \cite{BH92}. See Theorem \ref{BH92} below. 

\begin{theor} Let $H\leq Out(F_N)$ and suppose that $H$ contains an iwip $\varphi$. Then one of the following holds:
\begin{enumerate} 
\item $H$ contains a hyperbolic iwip. 
\item $H$ is geometric, i.e. $H$ contains no hyperbolic iwips and $H\leq Mod^{\pm}(S)\leq Out(F_N)$ where $S$ is a compact surface with one boundary component with $\pi_1(S)=F_N$ such that $\varphi\in H$ is induced by a pseudo-Anosov homeomorphism of $S$.  
\end{enumerate}
\end{theor}
Moreover, if the original iwip $\varphi\in H$ is non-atoroidal and $(1)$ happens, then $H$ contains a free subgroup $L$ of rank two such that every nontrivial element of $L$ is a hyperbolic iwip. See Remark \ref{ranktwo} below. 

In \cite{CFKM,CFKM2}, using a result of Handel-Mosher \cite{HM} about classification of subgroups of $Out(F_N)$, Carette-Francaviglia-Kapovich and Martino showed that every nontrivial normal subgroup of $Out(F_N)$ contains an iwip for $N\ge3$. And they asked whether every such subgroup contains a hyperbolic iwip. As a corollary of Theorem D we answer this question in the affirmative direction:

\begin{corol} Let $N\ge3$. Then, every nontrivial normal subgroup $H$ of $Out(F_N)$ contains a hyperbolic iwip.
\end{corol}

Recall that the \emph{minimal set} $\mathcal{M}_N\subset\mathbb{P}Curr(F_{N})$, introduced by R. Martin \cite{Martin}, is the closure of the set
\[
\{[\eta_{g}]\ |\ g\in F_{N}\ is\ primitive\ element\}
\]
in $\mathbb{P}Curr(F_{N})$. By a result of Kapovich-Lustig  \cite{KL1},  $\mathcal{M}_N$ is the unique smallest non-empty closed $Out(F_N)$-invariant subset of $\mathbb{P}Curr(F_N)$. Concretely, $\mathcal{M}_N$ is equal to the closure of the $Out(F_{N})$ orbit of $[\eta_{g}]$ for a primitive element $g\in F_{N}$. Note that for every non-atoroidal iwip $\varphi\in Out(F_N)$, its stable current $[\mu_+]$ belongs to $\mathcal{M}_N$. Indeed, for every primitive element $g\in F_N$ the positive iterates $\varphi^{n}([\eta_g])$ converge to $[\mu_{+}]$ by Theorem \ref{A}, and therefore $[\mu_+]\in\mathcal{M}_N$. For similar reasons $[\mu_{-}]\in\mathcal{M}_N$. As a direct consequence of Theorem \ref{B} we obtain:

\begin{corol}\label{minimalint}
Let $\varphi\in Out(F_{N})$ be a non-atoroidal iwip, the action of $\varphi$ on $\mathcal{M}_N$ has uniform north-south dynamics. In other words, given a compact set $K_0\subset\mathcal{M}_N\setminus\{[\mu_{-}]\}$ and an open neighborhood $U$ of  $[\mu_{+}]$ in $\mathcal{M}_N$, there is an integer $M_0>0$ such that $\varphi^{n}(K_0)\subset U$ for all $n\ge M_0$. Similarly, given a compact set $K_1\subset\mathcal{M}_N\setminus\{[\mu_{+}]\}$ and an open neighborhood $V$ of $[\mu_{-}]$ in $\mathcal{M}_N$, there is an integer $M_1>0$ such that $\varphi^{-m}(K_1)\subset V$ for all $m\ge M_1$.
\end{corol} 

In his 1995 thesis \cite{Martin}, Reiner Martin stated Corollary F, but without proof.
In \cite[Proposition 4.11]{BF10} Bestvina and Feighn prove uniform north-south dynamics for the action of a hyperbolic iwip on $\mathbb{P}Curr(F_N)$ and for a non-atoroidal iwip on $\mathcal{M}_N$ (i.e. Corollary F). Their proof follows Martin's original approach for the case of hyperbolic iwips which is quite different from our proof of Corollary F here. We also give a detailed proof of Reiner Martin's result in \cite{U}. 

\begin{ack*} I am indebted to my advisors Ilya Kapovich and Chris Leininger for guidance, support and encouragement. I am grateful to Matt Clay for bringing \cite{CP} to our attention, which helped us to streamline our original proof of Theorem \ref{mainapp}. I would also like to thank Spencer Dowdall, Martin Lustig, Brian Ray and Ric Wade for helpful discussions. Finally, I would like to thank to the referee for useful comments. 
\end{ack*}

\section{Preliminaries}\label{Prelim}
\subsection{Geodesic Currents on Free Groups and Outer Space}\label{2.1}	
Let $F_{N}$ be a finitely generated free group of rank $N\geq2$. Let
us denote the Gromov boundary of $F_{N}$ by $\partial F_{N}$ and set 
\[
\partial^{2}F_{N}:=\{(\xi,\zeta)\ | \xi,\zeta\in\partial F_{N},\ and\ \xi\neq\zeta\}.
\]

A geodesic current on $F_{N}$ is a positive locally finite Borel measure on $\partial^{2}F_{N}$,
which is $F_{N}$-invariant and $\sigma_{f}$-invariant, where $\sigma_{f}:\partial^{2}F_{N}\rightarrow\partial^{2}F_{N}$
is the $flip$ map defined by 
\[
\sigma_{f}(\xi,\zeta)=(\zeta,\xi)
\]
for $(\xi,\zeta)\in\partial^{2}F_{N}$. 
We will denote the space of geodesic currents on $F_{N}$ by $Curr(F_{N})$. The space $Curr(F_{N})$ is endowed with the weak* topology so that, given $\nu_{n},\nu\in Curr(F_{N})$, $\lim_{n\to\infty}\nu_{n}=\nu$ if and only if $\lim_{n\to\infty}\nu_{n}(S_1\times S_2)=\nu(S_1\times S_2)$ for all disjoint closed-open subsets $S_{1}, S_{2}\subseteq \partial F_{N}$. 

For a Borel subset $S$ of $\partial^{2}F_{N}$ and $\varphi\in Aut(F_N)$,
\[
\varphi\nu(S):=\nu(\varphi^{-1}(S))
\]
defines a continuous, linear left action of $Aut(F_{N})$ on $Curr(F_{N})$. Moreover, $Inn(F_{N})$ acts trivially, so that the action induces an action by the quotient group $Out(F_{N})$.

Let $\nu_{1},\nu_{2}$ be two non-zero currents, we say $\nu_{1}$
is equivalent to $\nu_{2}$, and write $\nu_{1}\sim\nu_{2}$, if there
is a positive real number $r$ such that $\nu_{1}=r\nu_{2}$. Then, the
space of \emph{projectivized geodesic currents} on $F_{N}$ is defined by
\[
\mathbb{P}Curr(F_{N}):=\{\nu\in Curr(F_{N}):\ \nu\neq0\}/\sim.
\]
We will denote the projective class of the current $\nu$ by $[\nu]$. 
The space $\mathbb{P}Curr(F_{N})$ inherits the quotient topology and the above $Aut(F_{N})$ and $Out(F_{N})$ actions on $Curr(F_{N})$ descend to well defined actions on $\mathbb{P}Curr(F_{N})$ as follows: For $\varphi\in Aut(F_{N})$ and $[\nu]\in\mathbb{P}Curr(F_{N})$,
\[
\varphi[\nu]:=[\varphi\nu].
\]

Let $R_N$ be the \emph{rose} with N petals, that is, a finite graph with one vertex $q$, and $N$ loop-edges attached to the vertex $q$. Let $\Gamma$ be a finite, connected graph with no valence-one vertices such that $\pi_{1}(\Gamma,p)\cong F_N$. A \emph{simplicial chart} or \emph{marking} on $F_{N}$ is a homotopy equivalence $\alpha:R_N\to\Gamma$. The marking $\alpha:R_N\to\Gamma$ induces an isomorphism, which we will also denote by $\alpha$, $\alpha:\pi_1(R_N,q)\to\pi_1(\Gamma,p)$ where $p=\alpha(q)$.  There are canonical $F_{N}$-equivariant homeomorphisms $\tilde{\alpha}=\partial\alpha:\partial F_{N}\to\partial\tilde{\Gamma}$ and $\partial^{2}\alpha:\partial^{2}F_{N}\to\partial^{2}\tilde{\Gamma}$ induced by $\alpha$.

Let $\gamma$ in $\tilde{\Gamma}$ be a reduced edge path. Define the \emph{cylinder} set  for $\gamma$ to be 
\[
Cyl_{\alpha}(\gamma):=\{(\xi,\zeta)\in\partial^{2}F_{N} : \gamma\subset[\tilde\alpha(\xi),\tilde\alpha(\zeta)]\},
\]
where $[\tilde\alpha(\xi),\tilde\alpha(\zeta)]$ denotes the geodesic from $\tilde\alpha(\xi)$ to $\tilde\alpha(\zeta)$ in $\tilde{\Gamma}$.

For an edge-path $v$ in $\Gamma$, let $\gamma$ be any lift of $v$ to $\tilde{\Gamma}$. We set 
\[
\left<v,\nu\right>_\alpha:=\nu\left(Cyl_{\alpha}(\gamma)\right)
\]
and call $\left<v,\nu\right>_{\alpha}$ \emph{the number of occurrences of $v$ in $\nu$}.

In \cite{Ka2}, Kapovich showed that given $\nu_{n},\nu\in Curr(F_{N})$, $\lim_{n\to\infty}\nu_{n}=\nu$ if and only if $\lim_{n\to\infty}\left<v,\nu_{n}\right>_\alpha=\left<v,\nu\right>_\alpha$ for all edge paths $v$ in $\Gamma$. Moreover a geodesic current is uniquely determined by the set of values $(\left<v,\nu\right>_{\alpha})_{v\in \mathcal{P}\Gamma}$, where $\mathcal{P}\Gamma$ is the set of all edge paths in $\Gamma$. 

For any $\nu\in Curr(F_{N})$, the \emph{weight of $\nu$ with respect to $\alpha$} is defined as 
\[
w_{\alpha}(\nu):=\sum_{e\in E\Gamma}\left<e,\nu\right>_{\alpha},
\]    
where $E\Gamma$ is the set of all oriented edges of $\Gamma$. It turns out that the concept of \emph{weight} is useful in terms of giving convergence criteria in the space of projective currents $\mathbb{P}Curr(F_{N})$:
\begin{lem}\cite{Ka2} For any nonzero $\nu_{n},\nu\in Curr(F_{N})$, we have
\[
\lim_{n\to\infty}[\nu_{n}]=[\nu]\ if\ and\ only\ if\  \lim_{n\to\infty}\frac{\nu_{n}}{w_{\alpha}(\nu_{n})}=\frac{\nu}{w_{\alpha}(\nu)}.
\]
\end{lem}
 
\begin{defn}[Rational Currents] Let $g\in F_N$ be a nontrivial element such that $g\neq h^{k}$ for any $h\in F_N$ and $k>1$. Then define the \emph{counting current} $\eta_{g}$ as follows:  For a closed-open subset $S$ of $\partial^{2} F_{N}$, $\eta_{g}$ is the number of $F_{N}$-translates of $(g^{-\infty}, g^{\infty})$ and $(g^{\infty}, g^{-\infty})$ that are contained in $S$. For an arbitrary non-trivial element $g\in F_N$ write $g=h^{k}$, where $h$ is not a proper power. Then define $\eta_{g}:=k\eta_{h}$. Any nonnegative scalar multiple of a counting current is called a \emph{rational current}. 
\end{defn}
An important fact about rational currents is that, the set of rational currents is dense in $Curr(F_{N})$, see \cite{Ka1,Ka2}. Note that for any $h\in F_N$ we have $(hgh^{-1})^{-\infty}=hg^{-\infty}$ and $(hgh^{-1})^{\infty}= hg^{\infty}$. From here, it is easy to see that $\eta_{g}$ depends only on the conjugacy class of the element $g$. So, from now on, we will use $\eta_{g}$ and $\eta_{[g]}$ interchangeably. 

It turns out that \emph{the number of occurrences} of an edge path in a \emph{rational current} has a more concrete description: Let $c$ be a circuit in $\Gamma$. For any edge-path $v$ define \emph{number of occurrences of $v$ in $c$}, denoted by $\left<v,c\right>_{\alpha}$, to be the number of vertices in $c$ such that starting from that vertex, moving in the positive direction on $c$ one can read off $v$ or $v^{-1}$ as an edge-path. Then for an edge-path $v$, and a conjugacy class $[g]$ in $F_{N}$ one has 
\[
\left<v,\eta_{[g]}\right>_{\alpha}=\left<v,c(g)\right>_{\alpha},
\]
where $c(g)=\alpha(g)$ is the unique reduced circuit in $\Gamma$ representing $[g]$, see \cite{Ka2}.

\begin{defn}[Outer Space]  The space of minimal, free and discrete isometric actions of $F_{N}$ on $\mathbb{R}$-trees up to $F_{N}$-equivariant isometry is denoted by $cv_N$ and called \emph{unprojectivized Outer Space}. The closure of the Outer Space, $\overline{cv}_N$, consists precisely of \emph{very small, minimal, isometric} actions of $F_{N}$ on $\mathbb{R}$-trees, see \cite{BF93}. 
It is known that \cite{CM} every point in the closure of the outer space is uniquely determined by its \emph{translation length function} $\|.\|_{T}:F_{N}\to\mathbb{R}$ where $\|g\|_{T}=\min\limits_{x\in T}d_{T}(x,gx)$. There is a natural continuous right action of $Aut(F_{N})$ on $\overline{cv}_N$, which in the level of translation length functions defined by
\[
\|g\|_{T\varphi}=\|\varphi(g)\|_{T}
\]
for any $T\in\overline{cv}_N$ and $\varphi\in Aut(F_{N})$. It is easy to see that for any $h\in F_{N}$, $\|hgh^{-1}\|_{T}=\|g\|_{T}$. So $Inn(F_{N})$ is in the kernel of this action, hence the above action factors through $Out(F_{N})$.
\end{defn}
A useful tool relating geodesic currents to outer space is the intersection form introduced by Kapovich-Lustig.
\begin{prop-defn}\cite{KL2}\label{intform} There exists a unique continuous map $\left<,\right>:\overline{cv}_N\times Curr(F_{N})\to\mathbb{R}_{\geq0}$ with the following properties:
\begin{enumerate}
\item $\left<T,c_{1}\nu_{1}+c_{2}\nu_{2}\right>=c_{1}\left<T,\nu_{1}\right>+c_{2}\left<T,\nu_{2}\right>$ for any $T\in\overline{cv}_N$, $\nu_{1},\nu_{2}\in Curr(F_{N})$ and nonnegative scalars $c_{1},c_{2}$. 
\item$\left<cT,\nu\right>=c\left<T,\nu\right>$ for any $T\in\overline{cv}_N$ and $\nu\in Curr(F_{N})$ and $c\geq0$. 
\item$\left<T\varphi,\nu\right>=\left<T,\varphi\nu\right>$ for any $T\in\overline{cv}_N$, $\nu\in Curr(F_{N})$ and $\varphi\in Out(F_{N})$.
\item $\left<T,\eta_{g}\right>=\|g\|_{T}$ for any  $T\in\overline{cv}_N$, any nontrivial $g\in F_{N}$.
\end{enumerate}
\end{prop-defn}
\noindent A detailed discussion of geodesic currents on free groups can be found
in \cite{Ka1,KL1,KL2,KL3}.
\subsection{Geodesic Currents on Surfaces \label{2.2}}
Let $S$ be a closed surface of genus $g\geq2$. Let us fix a hyperbolic metric on $S$ and consider the universal cover $\tilde{S}$ with the pull-back metric. There is a natural $\pi_{1}(S)$ action on $\tilde{S}$ by isometries. We will denote the space of bi-infinite, unoriented, unparameterized geodesics in $\tilde S$ by $G(\tilde S)$, given the quotient topology from the compact open topology on parameterized bi-infinite geodesics. Since such a geodesic is uniquely determined by the (unordered) pair of its distinct end points on the boundary at infinity of $\tilde{S}$, a more concrete description can be given by
\[
G(\tilde{S})=\big((\tilde{S}_{\infty}\times\tilde{S}_{\infty})-\Delta\big)\big/ \big((\xi,\zeta)=(\zeta,\xi)\big).  
\]

A \emph{geodesic current} on $S$ is a locally finite Borel measure on $G(\tilde{S})$ which is $\pi_{1}(S)$ invariant. The set of all geodesic currents on $S$, denoted by $Curr(S)$, is a metrizable topological space with the weak* topology, see \cite{Bo86,Bo88}.

As a simple example, consider the preimage in $\tilde{S}$ of any closed curve $\gamma\subset S$, which is a collection of complete geodesics in $\tilde{S}$. This is a discrete subset of $G(\tilde{S})$ which is invariant under the action of $\pi_{1}(S)$. Dirac (counting) measure associated to this set on $G(\tilde{S})$ gives a geodesic current on $S$, which is denoted by $\mu_\gamma$. Note that this construction gives an injection from the set of closed curves on $S$ to $Curr(S)$. As another example, one can consider a measured geodesic lamination as a geodesic current, see \cite{Bo86, Lein}. Therefore, the set $\mathcal{ML}(S)$ of all measured geodesic laminations on $S$, is a subset of $Curr(S)$. 

Recall that \emph{geometric intersection number} of $\alpha$, $\beta$ for any two homotopy classes of closed curves is the minimum number of intersections of $\alpha'$ and $\beta'$ for $\alpha'\simeq\alpha$ and $\beta'\simeq\beta$. One can show that this minimum is realized when $\alpha'$, $\beta'$ are geodesic representatives. 

In \cite{Bo86}, Bonahon constructed a continuous, symmetric, bilinear function $i(,):Curr(S)\times Curr(S)\to\mathbb{R}_{\geq0}$ such that  for any two homotopy classes of closed curves $\alpha$, $\beta$, $i(\mu_\alpha,\mu_\beta)$ is the geometric intersection number between $\alpha$ and $\beta$. We will call $i(,)$ the \emph{intersection number function}. 

We say that a geodesic current $\mu\in Curr(S)$ is \emph{binding} if $i(\nu,\mu)>0$ for all $\nu\in Curr(S)$. For example, let $\beta=\{\beta_1,\dotsc,\beta_m\}$ be a filling set of simple closed curves, i.e. union of geodesic representatives of $\beta_i$ cuts $S$ up into topological disks. Then $\mu_\beta=\sum\mu_{\beta_i}$ is a binding current. One of the facts proved in \cite{Bo88} that we will use repeatedly in our arguments is as follows:

\begin{prop}[Bonahon]\label{binding} Let $\beta$ be a binding current. Then the set of $\nu\in Curr(S)$ with $i(\nu,\beta)\leq M$ for $M>0$ is compact in $Curr(S)$. 
\end{prop}
Proposition \ref{binding} also implies that the \emph{space of projective geodesic currents}, $\mathbb{P}Curr(S)$, is compact where 
\[
\mathbb{P}Curr(S)=(Curr(S)\setminus0\big)/\sim
\] 
and where $\nu_{1}\sim\nu_{2}\ if\ and\ only\ if\ \nu_{1}=c\nu_{2}$ for some $c>0$.

In \cite{Otal}, Otal used the intersection number function to distinguish points in $Curr(S)$.
\begin{thm}\cite{Otal}\label{Otal} Given $\nu_{1},\nu_{2}\in Curr(S)$, $\nu_{1}=\nu_{2}\ if\ and\ only\ if\ i(\nu_{1},\alpha)=i(\nu_{2},\alpha)$ for every closed curve $\alpha$ in $S$.
\end{thm}
\noindent Recently, in \cite{DLR}, Duchin-Leininger-Rafi used Theorem \ref{Otal} to construct a metric on $Curr(S)$, which will be crucial in our argument. 
\begin{thm}\cite{DLR} Let $\beta$ be a filling set of simple closed curves. Enumerate all the closed curves on $S$ by $\{\alpha_{1},\alpha_{2},\alpha_{3}\dotsc\}$. Let $x_{k}=\dfrac{\alpha_{k}}{i(\alpha_{k},\beta)}$, then
\[
d(\nu_{1},\nu_{2})=\sum\limits_{k=1}^{\infty}\frac{1}{2^{k}}|i(\nu_{1},x_{k})-i(\nu_{2},x_{k})|
\]
defines a proper metric which is compatible with the weak* topology. 
\end{thm}
In general one can define a geodesic current on an oriented surface $S$ of finite type, but here we restrict our attention to compact surfaces with $b$ boundary components and give the definition for this particular case. Given a compact surface $S$ with $b$ boundary components, think of it as a subset of a complete, hyperbolic surface $S'$, obtained from $S$ by attaching $b$ \emph{flaring ends}, see Figure \ref{Flare}.  

\begin{figure}[h]
\centering
\includegraphics[trim=1cm 2cm 6cm 4cm, clip=true, totalheight=0.4\textheight, angle=95]{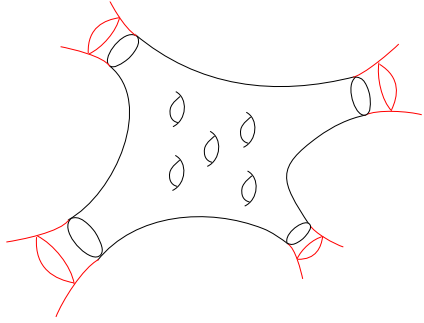}
\caption{Attaching flaring ends\label{Flare}}
\end{figure}

A \emph{geodesic current} on $S$ is a geodesic current on $S'$, a locally finite Borel measure on $G(\tilde{S'})$ which is $\pi_{1}(S')$ invariant, with the property that  support of this measure projects into $S$. 

An alternative definition can be given as follows: Let $S$ be an oriented surface of genus $g\geq1$ with $b\ge1$ boundary components. Consider the double of $S$, $DS$, which is a closed, oriented surface obtained by gluing two copies of $S$ along their boundaries by the identity map. We equip $DS$ with a hyperbolic metric such that $S$ has geodesic boundary. A geodesic current on $S$ can be thought as a geodesic current on $DS$ whose support projects entirely into $S$. Building upon work of Bonahon \cite{Bo86}, Duchin-Leininger-Rafi \cite{DLR} showed that intersection number on $Curr(S)$ is in fact the restriction of the intersection number defined on $Curr(DS)$. \\ For a detailed discussion of geodesic currents on surfaces we refer reader to \cite{Bo86,Bo88,DLR}.


\section{Dynamics on Surfaces}

Let $S$ be a compact surface of \emph{genus} $g\ge1$ with $b\ge1$ boundary components and let $f:S\rightarrow S$ be a \emph{pseudo-Anosov} map which fixes $\partial S$ pointwise. Let $\varphi\in Out(\pi_1(S))=Out(F_N)$ be the induced outer automorphism. Let us denote the \emph{stable measured lamination} corresponding to $f$ by $\mu_{+}$ and the \emph{unstable measured lamination} by $\mu_{-}$.  We will always assume that $i(\mu_{-},\mu_{+})=1$. Also fix a hyperbolic metric on $S$ and denote the geodesic boundary components of $S$ by $\alpha_1,\alpha_2,\dotsc\alpha_b$. We will extend the map $f$ on $S$  to double surface $DS$ by taking the identity map on the other side and will continue to denote our extension to $DS$ by $f$. Also we will adhere to the point of view that a geodesic current on $S$ is a geodesic current on $DS$ whose support projects entirely into $S$. Let $\mu_{\alpha_{i}}$ be the current corresponding to the boundary curve $\alpha_i$. 

Let us define $\Delta, H_{0}(f), H_{1}(f)\subset\mathbb{P}Curr(S)$ as follows:
\[
\Delta:=\{[a_1\mu_{\alpha_1}+a_2\mu_{\alpha_2}+\dotsc+a_b\mu_{\alpha_b}]\ |\ a_i\ge0,  \sum_{i=1}^{b}a_i>0 \}.
\]
\[
H_{0}(f):=\{[t_1\mu_{-}+t_2\nu]\ | [\nu]\in\Delta\}\ and\ H_{1}(f):=\{[t'_1\mu_{+}+t'_2\nu]\ |  [\nu]\in\Delta\}.
\]

First, we will prove a general result about geodesic currents which
will be useful throughout this article. 
\begin{prop}\label{Hset}
Let $f$ be a pseudo-Anosov homeomorphism on a surface $S$ with $b$ boundary components $\{\alpha_1,\alpha_2,\dotsc,\alpha_b\}$. Let $[\mu_{+}]$ and $[\mu_{-}]$ be the corresponding stable and unstable laminations for $f$.  Assume that $\nu\in Curr(S)$ is a geodesic current on $S$ such that $i(\nu,\mu_{+})=0$. Then,
\[
 \nu=a_{1}\mu_{\alpha_1}+a_2\mu_{\alpha_2}+\dotsc+a_b\mu_{\alpha_b}+c\mu_{+}
\]
for some $a_{i}, c\geq0$. Similarly, if $i(\nu,\mu_{-})=0$, then
\[
\nu=a'_{1}\mu_{\alpha_1}+a'_2\mu_{\alpha_2}+\dotsc+a'_b\mu_{\alpha_b}+c'\mu_{-}
\]
for some $a'_{i}, c'\geq0$.
\end{prop}

\begin{proof}
We will prove the first part of the proposition. The proof of the second
part is similar. Here we use the structure of geodesic laminations and in
particular the structure of the stable and unstable laminations. For a
detailed discussion see \cite{CB,Th,Lein}.

Let $S$ be a surface with $b\ge1$ boundary components and $DS$ the double of $S$.  
Let $\mu_{+}$ be the stable lamination corresponding to the pseudo-Anosov map $f$. Assume that for a geodesic current $\nu$, the intersection number $i(\nu,\mu_{+})=0$.
We'll now investigate possible geodesics in the support of $\nu$. Since $i(\nu,\mu_{+})=0$, it follows  from the definition of Bonahon's intersection form that projection of any geodesic in the support of $\nu$ onto
$S$ cannot intersect the leaves of the lamination transversely. Indeed, such an intersection would contribute to intersection number positively. Therefore, for each geodesic $\ell$ in the support of $\nu$ one of the following happens:  
\begin{enumerate} 
\item $\ell$ projects onto a leaf of the stable lamination or a boundary curve, or
\item $\ell$ projects onto a geodesic that is disjoint from the leaves of the stable lamination or the boundary curves. 
\end{enumerate}
If (1) happens for every geodesic $\ell$ in the support of $\nu$, then there is nothing to prove.  If not, then $\ell$ projects to a geodesic, which is disjoint from the stable lamination. So if we cut $S$ along the stable lamination to get ideal polygons or crowns \cite{CB} (see Figure \ref{compregions}), then one of these contains the image of the geodesic $\ell$. 

\begin{figure}[h]
\centering
\includegraphics[trim=0cm 13cm 0cm 1cm, clip=true, totalheight=0.3\textheight, angle=0]{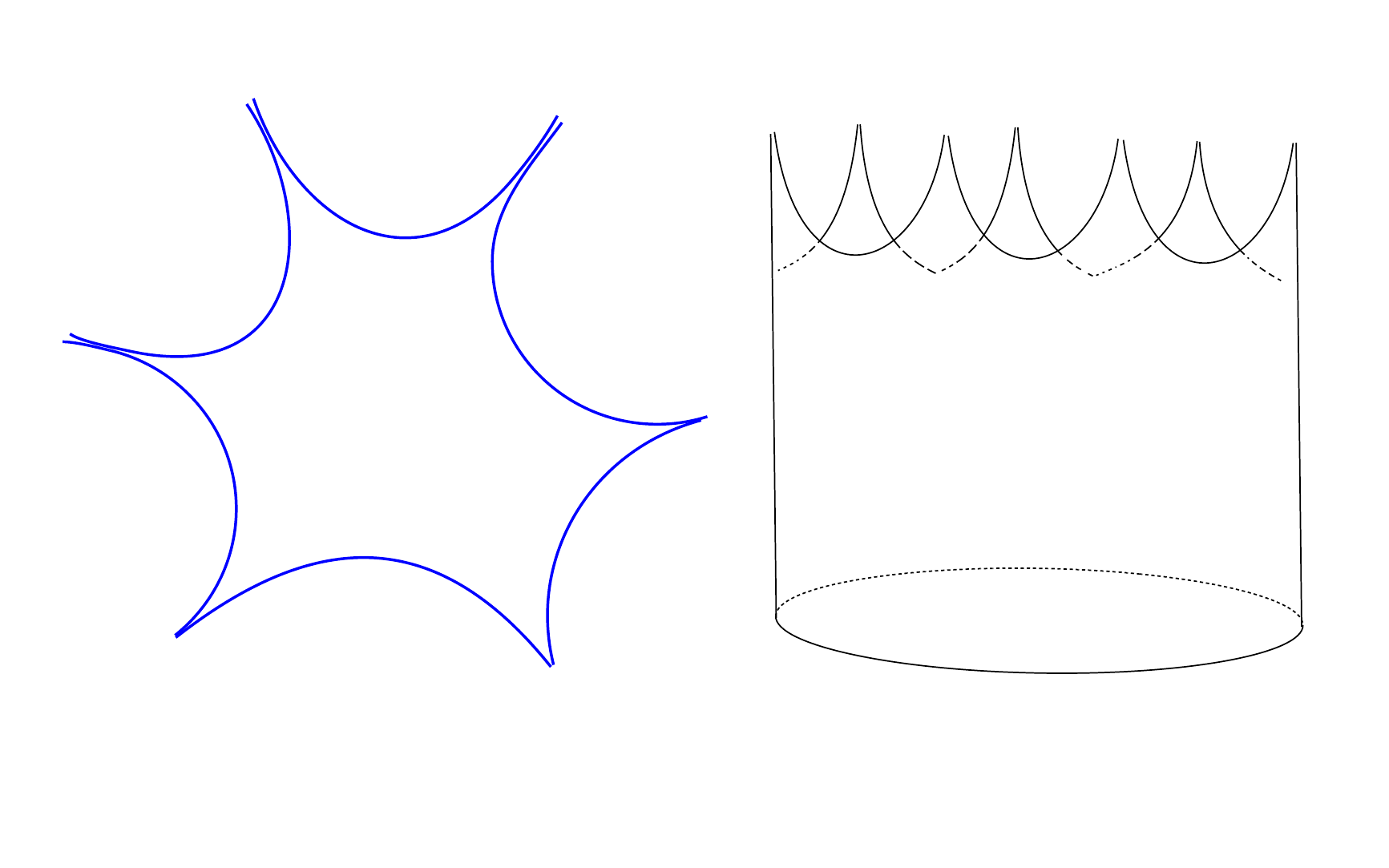}
\caption{Complementary Regions\label{compregions}}
\end{figure}

From here, it follows that there is a geodesic $\ell'$ (possibly different than $\ell$) in the support of $\nu$ which is \emph{isolated}, i.e. there is some open set $U$ in $G(\tilde{S})$, which intersects the support of $\nu$ in the single point $\{\ell'\}$.  The geodesic $\ell'$ must have positive mass assigned by $\nu$. This geodesic will project onto a biinfinite geodesic in $S$, which must therefore accumulate since $S$ is compact. Take a short geodesic segment $L$ in $S$, which the projection of $\ell'$ intersects infinitely often. Consider a lift $\tilde{L}$ of $L$ to $\tilde{S}$. There are infinitely many translates of $\ell'$ that intersect $\tilde{L}$ which also form a compact set. By the $\pi_1(S)$-invariance of the measures, this set must have infinite measure, which contradicts the fact that $\nu$ is a locally finite Borel measure. Therefore (2) does not happen. So as a result, $\nu$ will be a combination of some multiple
of the stable lamination and a linear combination of currents corresponding to the boundary curves. 
\end{proof}
\noindent The following lemma is due to N. Ivanov.
\begin{lem}\cite{Iva}\label{Ivanov} Let $\mu$ be a measured geodesic lamination on $DS$, if $[\mu]\notin H_{0}(f)$ then 
\[
\lim_{n\to\infty}\lambda^{-n}f^{n}(\mu)=i(\mu,\mu_{-})\mu_{+},
\]
where $\lambda$ is the dilatation.
Similarly if $[\mu]\notin H_{1}(f)$, then 
\[
\lim_{n\to\infty}\lambda^{-n}f^{-n}(\mu)=i(\mu,\mu_{+})\mu_{-}.
\]
\end{lem}
The first step toward a generalization of the Lemma \ref{Ivanov} for actions of pseudo-Anosov homeomorphisms on the space of projective geodesic currents can be stated as follows: 

\begin{lem}\label{projconv}
Let $[\nu]\notin H_{0}(f)$ be an arbitrary geodesic current, then 
\[
\lim_{n\to\infty}f^{n}[\nu]=[\mu_{+}].
\]
\end{lem}
\begin{proof} Considering $\mathbb{P}Curr(S)\subset\mathbb{P}Curr(DS)$ as a closed subset, since $\mathbb{P}Curr(DS)$ is compact, there exist a subsequence $\{n_{k}\}$ such that $\lim_{n_{k}\to\infty}f^{n_{k}}[\nu]$ exists. This means that there is a sequence $\{w_{n_{k}}\}$ of positive real numbers such that
\[
\lim_{n_{k}\to\infty}w_{n_{k}}f^{n_{k}}(\nu)=\nu_{*}\ne0
\]
where $\nu_*\in\mathbb{P}Curr(S)$. 
Let $\beta$ be a filling set of simple closed curves in $DS$. We have
\begin{align*} 0\ne C & = i(\nu_{*},\beta)\\ & =i(\lim_{n_{k}\to\infty}w_{n_{k}}f^{n_{k}}(\nu),\beta) \\
 & =\lim_{n_{k}\to\infty}i(w_{n_{k}}\nu,f^{-n_k}(\beta))\\
 & =\lim_{n_{k}\to\infty}w_{n_{k}}\lambda^{n_{k}}i(\nu,\lambda^{-n_{k}}f^{-n_{k}}(\beta))\\
 & =\lim_{n_{k}\to\infty}w_{n_{k}}\lambda^{n_{k}}\lim_{n_{k}\to\infty}i(\nu,\lambda^{-n_{k}}f^{-n_{k}}(\beta)))\\
 & =\lim_{n_{k}\to\infty}w_{n_{k}}\lambda^{n_{k}}i(\nu,\mu_{-})i(\mu_{+},\beta).
\end{align*}
From here we can deduce that
\[
\lim_{n_{k}\to\infty}w_{n_{k}}\lambda^{n_{k}}=C_{1}\neq0,
\]
since $i(\nu,\mu_{-})\neq0$ and $i(\mu_+,\beta)\neq0$. This means that without loss of generality we can choose $w_{n_{k}}=\lambda^{-n_{k}}$. Now, look at 
\begin{align*} i(\nu_{*},\mu_{+}) & =i(\lim_{n_{k}\to\infty}\lambda^{-n_{k}}f^{n_{k}}(\nu),\mu_{+})\\
 & =\lim_{n_{k}\to\infty}i(\nu,\lambda^{-n_{k}}f^{-n_{k}}(\mu_{+}))\\
 & =\lim_{n_{k}\to\infty}\lambda^{-2n_{k}}i(\nu,\mu_{+})\\
 & =0.
\end{align*}
Therefore, by Proposition \ref{Hset}, 
\[
\nu_{*}=t\mu_{0}+s\mu_{+}
\] where $t,s\ge0,\ t+s>0$ and $\mu_{0}$ is a non-negative, non-trivial linear combination of currents corresponding to boundary curves. 
\begin{claim} $s\neq0$, i.e $\nu_{*}\neq t\mu_{0}$.
\end{claim}
We will show that $\nu_{*}$ has non-zero intersection number with $\mu_{-}$ which will imply that $\nu_{*}\neq\mu_{0}$. 
\begin{align*}i(\nu_{*},\mu_{-}) & =i(\lim_{n_{k}\to\infty}\lambda^{-n_{k}}f^{n_{k}}(\nu),\mu_{-})\\
 & =\lim_{n_{k}\to\infty}i(\nu,\lambda^{-n_{k}}f^{-n_{k}}(\mu_{-}))\\
 & =\lim_{n_{k}\to\infty}i(\nu,\mu_{-})\\
 & =i(\nu,\mu_{-})\neq0,
\end{align*}
since $\nu\notin H_{0}$. 
\begin{claim} $\lim_{n\to\infty}f^{n}[\nu]=[\mu_{+}]$.
\end{claim}
\noindent Assume not, then there exists a subsequence $\{n_{k}\}$ such that 
\[
\lim_{n_{k}\to\infty}f^{n_{k}}[\nu]=[\nu_{*}]\neq[\mu_{+}].
\]
First observe that
\[
\lim_{m\to\infty}f^{-m}[\nu_{*}]=[\mu_{0}],
\]
since 
\[
f^{-m}[t\mu_{0}+s\mu_{+}]=[t\mu_{0}+\lambda^{-m}s\mu_{+}].
\]
We also have,  $\forall i>0$
\[ \lim_{n_{k}\to\infty}f^{n_{k}-i}[\nu]=f^{-i}[\nu_{*}].
\]
Let  $d'$ be a metric on $\mathbb{P}Curr(DS)$ which induces the quotient topology from the weak-* topology.
Then, $\forall m\geq1$ there exist an integer $i_{m}\ge1$ such that 
\[
d'([\mu_{0}],f^{-i_{m}}[\nu_{*}])<\frac{1}{2m}.
\]
Since $\lim_{n_{k}\to\infty}f^{n_{k}-i_{m}}[\nu]=f^{-i_{m}}[\nu_{*}]$, 
pick $n_{k}=n_{k}(m)$ such that 
\[
d'(f^{-i_{m}}[\nu_{*}],f^{n_{k}-i_{m}}[\nu])<\frac{1}{2m}\ and\  n_{k}-i_{m}>m.
\]
Now, set $j_{m}=n_{k}-i_{m}$, then by triangle inequality we have 
\[
d'(f^{j_{m}}[\nu],[\mu_{0}])<\frac{1}{m},
\]
which implies 
$\lim_{j_{m}\to\infty}f^{j_{m}}[\nu]=[\mu_{0}]$, which contradicts with the previous claim.   
This completes the proof of the lemma.   
\end{proof}

\begin{thm}\label{PNS}
Let $f$ be a pseudo-Anosov homeomorphism on a compact surface $S$. Let $[\nu]\notin H_{0}(f)$ be a geodesic current. Then, 
\[
\lim_{n\to\infty}\lambda^{-n}f^{n}(\nu)=i(\nu,\mu_{-})\mu_{+},
\]
where $\mu_{+}$ and $\mu_{-}$   are the stable and the unstable laminations for $f$. Similarly, for a geodesic current $[\nu]\notin H_1(f)$, we have \[
\lim_{n\to\infty}\lambda^{-n}f^{-n}(\nu)=i(\nu,\mu_+)\mu_{-}.
\]
\end{thm}
\begin{proof}
Lemma \ref{projconv} implies that there exits a sequence $\{w_{n}\}$ of positive real numbers such that $\lim_{n\rightarrow\infty}w_{n}f^{n}(\nu)=\mu_{+}$. Recall that we assumed $i(\mu_{+},\mu_{-})=1$. 
Hence, we have 
\begin{align*}
1& =i(\lim_{n\rightarrow\infty}w_{n}f^{n}(\nu),\mu_{-})\\
 & =\lim_{n\rightarrow\infty}i(w_{n}\nu,f^{-n}(\mu_{-}))\\
 & =\lim_{n\rightarrow\infty}w_{n}\lambda^{n}i(\nu,\mu_{-}),
\end{align*}
which implies that $\lim_{n\to\infty}w_n\lambda^{n}=\dfrac{1}{i(\nu,\mu_-)}$. Therefore, we have
\[
\lim_{n\to\infty}\lambda^{-n}f^{n}(\nu)=\lim_{n\to\infty}\frac{w_n}{w_n\lambda^{n}}f^{n}(\nu)=i(\nu,\mu_{-})\mu_{+}.
\]
The proof of the second assertion is similar. 
\end{proof}
Now we are ready to prove the main theorem of this section.  
\begin{thm}\label{NSG}
Let $f$ be a pseudo-Anosov homeomorphism on a compact surface with boundary. Assume that $K$ is a compact set in $\mathbb{P}Curr(S)\setminus H_{0}(f)$. Then for any open neighborhood $U$ of the stable current $[\mu_{+}]$, there exist $m\in\mathbb{N}$ such that $f^{n}(K)\subset U$ for all $n\geq m$. Similarly, for a compact set $K'\subset\mathbb{P}Curr(S)\setminus H_{1}(f)$ and an open neighborhood $V$ of the unstable current $[\mu_{-}]$, there exist $m'\in\mathbb{N}$ such that $f^{-n}(K')\subset V$ for all $n\ge m'$.  
\end{thm}

\begin{proof} We will prove the first statement. The proof of the second one is similar. To prove this theorem we utilize the metric on the space of geodesic currents introduced by Duchin-Leininger-Rafi in \cite{DLR}. Let $d$ be the metric on $Curr(DS)$ as discussed in Section \ref{2.2}. Let $\{\alpha_0,\alpha_1,\alpha_2,\dotsc\}$ be an enumeration of all the closed curves on $DS$.  Let us set $x_{k}=\dfrac{\alpha_{k}}{i(\alpha_{k},\beta)}$ where $\beta$ is a filling set of simple closed curves on $DS$.  

Let us take a cross section $\bar{K}$ of $K$ in $Curr(S)$ by picking the representative $\nu\in  Curr(S)$ with $i(\nu,\mu_{-})=1$ for any $[\nu]\in K$.  From Theorem \ref{PNS}, we know that 
\[
\lim_{n\to\infty}\lambda^{-n}f^{n}(\nu)=i(\nu,\mu_{-})\mu_{+}=\mu_{+}
\]
So it suffices to show that for any $\epsilon>0$, there exist $m>0$ such that 
\[
d(\lambda^{-n}f^{n}(\nu),\mu_{+})<\epsilon
\]
for all $n\geq m$ and for all $\nu\in\bar{K}$. 

\begin{claim*}
\[
|i(\mu_{+},x_{k})-i(\lambda^{-n}f^{n}(\nu),x_{k})|
\]
is uniformly bounded $\forall\nu\in\bar{K}$, $\forall k\geq1$, $\forall n\geq1$.
\end{claim*}
\noindent By triangle inequality $|i(\mu_{+},x_{k})-i(\lambda^{-n}f^{n}(\nu),x_{k})|\le|i(\mu_{+},x_{k})|+|i(\lambda^{-n}f^{n}(\nu),x_{k})|$, so it suffices to bound the two quantities on the right. Since $\{x_{k}\}$ is precompact we have 
\[
i(\mu_{+},x_{k})\leq R_{0}
\]
for some $R_{0}>0$. \\
For the second quantity we have 
\[
i(\lambda^{-n}f^{n}(\nu),x_{k})=i(\nu,\lambda^{-n}f^{-n}(x_{k}))
\]
Since $\nu$ comes from a compact set, it suffices to show that $\{\lambda^{-n}f^{-n}(x_{k})\}$ is precompact. Now we have,
\[
i(\lambda^{-n}f^{-n}(x_{k}),\beta)=i(x_{k},\lambda^{-n}f^{n}(\beta))\leq R_{1}
\]
for some $R_{1}>0$ since $\{x_{k}\}$ is precompact and $\lim_{n\to\infty}\lambda^{-n}f^{n}(\beta)=i(\beta,\mu_{-})\mu_{+}$. Therefore by proposition \ref{binding}, $\{\lambda^{-n}f^{n}(x_{k})\}$ is precompact, and hence the second quantity is uniformly bounded, and the claim follows. 

By the claim there is some $R>0$ so that 
\[
|i(\mu_{+},x_{k})-i(\lambda^{-n}f^{n}(\nu),x_{k})|\leq R.
\]
Now we have;
\begin{equation}\label{distance}
d(\mu_{+},\lambda^{-n}f^{n}(\nu))\leq\sum\limits_{j=1}^M \frac{1}{2^{j}}|i(\mu_{+},x_{j})-i(\lambda^{-n}f^{n}(\nu),x_{j})|+\sum\limits_{j=M+1}^\infty \frac{1}{2^{j}}R. 
\end{equation}
The second sum can be made as small as we want by choosing $M$ big enough, so we make it less than $\epsilon/2$.  We need to show that first sum goes to 0 uniformly over all $\nu\in\bar{K}$. Since there are only finitely many terms we will show that for a fixed $x_{j}$  
\[
|i(\mu_{+},x_{j})-i(\lambda^{-n}f^{n}(\nu),x_{j})|\to 0
\]
uniformly $\forall\nu\in\bar{K}$. Now, because  $i(\nu,\mu_-)=1$, we have 
\begin{align*}|i(\mu_{+},x_{j})-i(\lambda^{-n}f^{n}(\nu),x_{j})| & =|i(\nu,\mu_{-})i(\mu_{+},x_{j})-i(\nu,\lambda^{-n}f^{-n}(x_{j}))|\\
& =|i(\nu,i(\mu_{+},x_{j})\mu_{-})-i(\nu,\lambda^{-n}f^{-n}(x_{j}))|.
\end{align*} 

Since $\nu$ comes from a compact set, we can choose a small neighborhood $V$ of $i(\mu_{+},x_{j})\mu_{-}$ such that for any $\nu'\in V$
\[
|i(\nu,i(\mu_{+},x_{j})\mu_{-})-i(\nu,\nu')|<\epsilon/2M
\]
for all $\nu\in\bar{K}$. 
We already know by Theorem \ref{PNS} that
\[
\lim_{n\to\infty}\lambda^{-n}f^{-n}(x_{j})=i(\mu_{+},x_{j})\mu_{-}.
\]
So there exist $m\in\mathbb{N}$ such that for all $n\ge m$ one has $\lambda^{-n}f^{-n}(x_{j})\in V$ and therefore
\[
|i(\nu,i(\mu_{+},x_{j})\mu_{-})-i(\nu,\lambda^{-n}f^{-n}(x_{j}))|<\epsilon/2M. 
\]
Hence, by \eqref{distance} we have
\[
d(\mu_{+},\lambda^{-n}f^{n}(\nu))\leq\epsilon
\]
for all $n\ge m$ and for all $\nu\in\bar{K}$.
\end{proof}

\section{Dynamics of non-atoroidal, fully irreducible automorphisms}

For a non-atoroidal iwip $\varphi\in Out(F_N)$, using a theorem of Bestvina and Handel we will be able to transfer the question about the dynamics of the action of $\varphi$ on $\mathbb{P}Curr(F_N)$ to a problem in surface theory. Using the result we established in the previous section,  we will prove a variant of uniform North-South dynamics on the space of geodesic currents for non-atoroidal iwips. 

The result of Bestvina-Handel we need is the following.
\begin{thm}\cite{BH92}\label{BH92} 
Let $\varphi\in Out(F_N)$. Then $\varphi$ is a \emph{non-atoroidal iwip} if and only if $\varphi$ is induced by a pseudo-Anosov homeomorphism $f$ of a compact surface $S$ with one boundary component and $\pi_1(S)\cong F_N$.
\end{thm}
\begin{rem} Note that with the  definition we gave at the end of Section \ref{Prelim}, a geodesic current on $F_{N}$ is precisely a geodesic current on a surface $S$ with $\pi_{1}(S)=F_{N}$. Therefore, $Curr(F_{N})=Curr(S)$. 
\end{rem}
For $\varphi$ and $f$ as in Theorem \ref{BH92}, define $H_0(\varphi):=H_0(f), H_1(\varphi)=H_1(f)\subset\mathbb{P}Curr(S)=\mathbb{P}Curr(F_N)$.
Combining the above remark and Theorem \ref{BH92}, a special case of Theorem \ref{NSG} for one boundary component gives the following:

\begin{thm} \label{NStoroidal} Let $\varphi\in Out(F_{N})$ be a non-atoroidal iwip. Then the action of $\varphi$ on the space of projectivized geodesic currents $\mathbb{P}Curr(F_{N})$ has uniform dynamics in the following sense: Given an open neighborhood $U$ of the stable current $[\mu_{+}]$ and a compact set $K_0\subset\mathbb{P}Curr(F_{N})\setminus H_{0}(\varphi)$ there exist a power $M_0>0$ such that for all $n\ge M_0$, $\varphi^{n}(K_0)\subset U$.  Similarly, given an open neighborhood $V$ of the unstable current $[\mu_{-}]$ and a compact set $K_1\subset\mathbb{P}Curr(F_{N})\setminus H_{1}(\varphi)$, there exist a power $M_1>0$ such that for all $m\ge M_1$, $\varphi^{-m}(K_1)\subset V$.
\end{thm} 

We complement the above theorem by giving a complete picture in terms of fixed points of the action of a non-atoroidal iwip $\varphi$ on the space of projective geodesic currents. 
\begin{prop}
Let $\varphi$ be a non-atoroidal iwip. Then the action of $\varphi$ on $\mathbb{P}Curr(F_{N})$ has exactly three fixed
points: the stable lamination (current) $[\mu_{+}]$, the unstable lamination (current) $[\mu_{-}]$
and the current corresponding to the boundary curve $[\mu_{\alpha}]$.
\end{prop}
\begin{proof}
A straightforward computation shows that points in $(H_{0}(\varphi)\cup H_{1}(\varphi))$ except $[\mu_+], [\mu_-]$ and $[\mu_\alpha]$ are not fixed. For any other point $[\nu]\in\mathbb{P}Curr(F_{N})\setminus\big(H_0(\varphi)\cup H_1(\varphi)\big)$, Theorem \ref{NStoroidal} and the fact that $[\mu_{+}]\neq[\mu_{-}]$ implies that $\varphi([\nu])\neq[\nu]$.  
\end{proof}

\noindent Recall that the minimal set $\mathcal{M}_N$ in $\mathbb{P}Curr(F_{N})$ is the closure of the set
\[
\{[\eta_{g}]\ |\ g\in F_{N}\ is\ primitive\ element\}
\]
in $\mathbb{P}Curr(F_{N})$. 
Equivalently, $\mathcal{M}_{N}$ is equal to the closure of $Out(F_{N})$ orbit of $[\eta_{g}]$ for a primitive element $g\in F_{N}$.  As a  consequence of Theorem \ref{NSG}, we obtain the following result, which was claimed without proof by R. Martin \cite{Martin}. A sketch of the proof following a different approach was given by Bestvina-Feign in \cite{BF10}. 

\begin{cor}
Let $\varphi\in Out(F_{N})$ be a non-atoroidal iwip  with stable and unstable currents $[\mu_{+}]$ and $[\mu_{-}]$. Then the action of $\varphi$ on $\mathcal{M}_N$ has uniform north-south dynamics. Namely, given a compact set $K_0\subset\mathcal{M}_N\setminus\{[\mu_{-}]\}$ and an open neighborhood $U$ of $[\mu_{+}]$ in $\mathcal{M}_N$, there is an integer $M_0>0$ such that $\varphi^{n}(K_0)\subset U$ for all $n\ge M_0$. Similarly, given a compact set $K_1\subset\mathcal{M}_N\setminus\{[\mu_{+}]\}$ and an open neighborhood $V$ of $[\mu_{-}]$ in $\mathcal{M}_N$, there is an integer $M_1>0$ such that $\varphi^{-m}(K_1)\subset V$ for all $m\ge M_1$.
\end{cor}
\begin{proof} Let $\varphi\in Out(F_{N})$ be a non-atoroidal iwip. Then $\varphi$ is induced by a pseudo-Anosov homeomorphism $f$ of a compact surface $S$ with one boundary component $\alpha$ and $\pi_1(S)\cong F_N$. Note that the current corresponding to the boundary curve $\mu_\alpha$ does not belong to minimal set $\mathcal{M}_N$. Indeed, it is well known that if a current $[\nu]\in\mathcal{M}_N$ and $A$ is free basis for $F_N$, then Whitehead graph of support of $[\nu]$ with respect to $A$ is either disconnected or connected but has a cut vertex, \cite{WH,BF10}. Pick a basis $A$ such that $\mu_\alpha$ corresponds to product of commutators. It is straightforward to check that $\mu_\alpha\notin\mathcal{M}_N$ by using this criteria. From here it also follows that any element in $(H_{0}\cup H_{1})$ other than $[\mu_{+}]$ and $[\mu_{-}]$ is not in $\mathcal{M}_N$ since the closure of the $\varphi$ orbit, and hence the $Out(F_{N})$ orbit, of any such element will contain $[\mu_{\alpha}]$. Therefore, Theorem \ref{NSG} implies that $\varphi$ has uniform north-south dynamics on the minimal set $\mathcal{M}_N$. 
\end{proof} 
As another corollary of our theorem we prove an analogous result for non-atoroidal iwips of a theorem of Kapovich-Lustig \cite{KL3} about hyperbolic iwips. Recall that, by \cite{LL}, the action of an iwip $\varphi\in Out(F_N)$ on the projectivized outer space $\overline{CV}_N$ has exactly two fixed points, $[T_{+}]$ and $[T_{-}]$, called attracting and repelling trees for $\varphi$. For any $[T]\ne[T_{-}]$, $\lim_{n\to\infty}[T]\varphi^{n}=[T_+]$ and for any $[T]\ne[T_+]$, $\lim_{n\to\infty}[T]\varphi^{-n}=[T_-]$. Moreover, there are constants $\lambda_{+},\lambda_{-}>1$ such that $T_{+}\varphi=\lambda_{+}T_{+}$ and $T_{-}\varphi^{-1}=\lambda_{-}T_{-}$. In fact, for a non-atoroidal iwip $\varphi\in Out(F_N)$, one has $\lambda_{-}=\lambda_{+}=\lambda$. 
\begin{thm}\label{Erg}
Let $\varphi\in Out(F_{N})$ be a non-atoroidal iwip. Let $T_{+}$ and $T_{-}$ be representatives of attracting and repelling trees, respectively, in $\overline{cv}_N$ corresponding to the right action of $\varphi$ on $\overline{cv}_N$. Then, 
\[
\left<T_{+},\mu\right>=0 \iff [\mu]=[a_{0}\mu_{-}+b_{0}\mu_{\alpha}]
\] for some $a_{0}\ge0,b_{0}\ge0.$ Similarly,
\[
\left<T_{-},\mu\right>=0 \iff [\mu]=[a_{1}\mu_{+}+b_{1}\mu_{\alpha}]
\] for some $a_{1}\ge0, b_{1}\ge0.$
\end{thm}
\begin{proof} We will prove the first assertion, the second one is symmetric. The \emph{``If"} direction follows from the properties of the intersection form (see Proposition \ref{intform}). Specifically, we have
\[
\left<T_{+},\mu_{-}\right>=\left<T_{+}\varphi,\varphi^{-1}\mu_{-}\right>
=\lambda\lambda_{+}\left<T_{+},\mu_{-}\right>,
\] which implies $\left<T_{+},\mu_{-}\right>=0$. Similarly, 
\[
\left<T_{+},\mu_{\alpha}\right>=\left<T_{+}\varphi,\varphi^{-1}\mu_{\alpha}\right>=\lambda_{+}\left<T_{+},\mu_{\alpha}\right>,
\] which implies that $\left<T_{+},\mu_{\alpha}\right>=0$ as well. Therefore, 
\[
\left<T_{+},a_{0}\mu_{-}+b_{0}\mu_{\alpha}\right>=a\left<T_{+},\mu_{-}\right>+b\left<T_{+},\mu_{\alpha}\right>=0.
\]
Conversely, let $\left<T_{+},\mu\right>=0.$ Assume that $\mu$ is not a linear combination of $\mu_{-}$ and $\mu_{\alpha}$. Then, there exist a sequence of positive real numbers $\{a_{n}\}$ such that\[
\lim_{n\to\infty}a_{n}\varphi^{n}(\mu)=\mu_{+}.
\]
Therefore by continuity of the intersection number we have 
\[
0\ne\left<T_{+},\mu_{+}\right>=\left<T_{+},\lim_{n\to\infty}a_{n}\varphi^{n}(\mu)\right>=\lim_{n\to\infty}a_{n}\lambda_{+}^{n}\left<T_{+},\mu\right>=0,
\] which is a contradiction.  
\end{proof}
The other direction of this unique-ergodicity type result is the same as for hyperbolic case. 
\begin{thm}\label{Erg2} Let $\varphi$ be a non-atoroidal iwip. Let $T_{-}, T_{+}$ be as in Theorem \ref{Erg}. Let $\mu_{+}$ and $\mu_{-}$ be representatives of  stable and unstable currents corresponding to action of $\varphi$ on $\mathbb{P}Curr(F_{N})$ accordingly. Then
\[
\left<T,\mu_{\pm}\right>=0\iff [T]=[T_{\mp}].
\]  
\end{thm} 

\begin{proof} We will prove that  $\left<T,\mu_{-}\right>=0\iff [T]=[T_{+}]$. The proof of the other assertion is similar. We have already proved in the previous theorem that $\left<T_{+},\mu_{-}\right>=0$. Let us assume that $\left<T,\mu_{-}\right>=0$ but $[T]\ne[T_{+}]$. Then, by \cite{LL}, there exist a sequence of positive real numbers $\{b_{n}\}$ such that
\[
\lim_{n\to\infty}b_{n}T\varphi^{-n}=T_{-}.
\]
Therefore, by continuity of the intersection number we get
\[
0\ne\left<T_{-},\mu_{-}\right>=\left<\lim_{n\to\infty}b_{n}T\varphi^{-n},\mu_{-}\right>=\left<\lim_{n\to\infty}b_{n}T,\varphi^{-n}\mu_{-}\right>\lim_{n\to\infty}\lambda^{n}b_{n}\left<T,\mu_{-}\right>=0,
\]
which is a contradiction. 
\end{proof}

\section{Subgroups of $Out({F}_{N})$ containing hyperbolic iwips}
\noindent The following lemma is an adaptation of Lemma 3.1 of \cite{CP}. 
\begin{lem} \label{cpsimilar} Let $\varphi\in Out(F_{N})$ be a non-atoroidal iwip. Let $[\mu_{-}],[\mu_{+}],[\mu_{\alpha}]$ be the unstable current, stable current and current corresponding to boundary curve respectively. Denote the convex hull of $[\mu_{-}]\ and\ [\mu_{\alpha}]$ by $H_{0}$ and the convex hull of  $[\mu_{+}]$ and $[\mu_{\alpha}]$ by $H_{1}$. Assume that $\psi\in Out(F_{N})$ is such that $\psi H_{1}\cap H_{0}=\emptyset$. Then there exist an integer $M\ge 1$ such that for all $m\ge M$, the element $\varphi^{m}\psi$ is hyperbolic.  
\end{lem}
\begin{proof} Recall that, since $\varphi$ is a non-atoroidal iwip, $\varphi$ is induced by a pseudo-Anosov $g\in Mod^{\pm}(S)$, where $S$ is a compact surface with single boundary component and $\pi_1(S)\cong F_N$. Therefore, $\lambda_{-}(\varphi)=\lambda_{+}(\varphi)=\lambda$, where $\lambda$ is the dilatation for $g$. Let $T_{+}$ and $T_{-}$ be representatives of the attracting and repelling trees for $\varphi$ in $\overline{cv}_N$ so that $T_{+}\varphi=\lambda T_{+}$ and $T_{-}\varphi^{-1}=\lambda T_{-}$. Then for all $m\ge0$ and $\nu\in Curr(F_{N})$
\[
\left<T_{+},\varphi^{m}\psi\nu\right>=\left<T_{+}\varphi^{m},\psi\nu\right>=\lambda^{m}\left<T_{+},\psi\nu\right>,
\] and
\[
\left<T_{-}\psi,\psi^{-1}\varphi^{-m}\nu\right>=\left<T_{-}\varphi^{-m},\nu\right>=\lambda^{m}\left<T_{-},\nu\right>.
\]
Now define 
\[
\alpha_{1}(\nu)=\max\{\left<T_{+},\nu\right>,\left<T_{-}\psi,\nu\right>\}
\] and 
\[
\alpha_{2}(\nu)=\max\{\left<T_{+},\psi\nu\right>,\left<T_{-},\nu\right>\}.
\]
Then,
\[
\alpha_{1}(\varphi^{m}\psi\nu)\ge\left<T_{+},\varphi^{m}\psi\nu\right>=\lambda^{m}\left<T_{+},\psi\nu\right>\] and \[
\alpha_{1}(\psi^{-1}\varphi^{-m}\nu)\ge\left<T_{-}\psi,\psi^{-1}\varphi^{-m}\nu\right>=\lambda^{m}\left<T_{-},\nu\right>
\] Hence
\[
\max\{\alpha_{1}(\varphi^{m}\psi\nu),\alpha_{1}(\psi^{-1}\varphi^{-m}\nu)\}\ge\lambda^{m}\alpha_{2}(\nu)
\]
Now $\alpha_{2}(\nu)=0$ if and only if $\left<T_{+},\psi\nu\right>=0$ and $\left<T_{-},\nu\right>=0$. By Theorem \ref{Erg} $\left<T_{-},\nu\right>=0\iff[\nu]\in H_{1}$. Since by assumption $\psi H_{1}\cap H_{0}=\emptyset$, this implies that $\left<T_{+},\psi\nu\right>\ne0$ again by Theorem \ref{Erg}. Therefore $\alpha_2(\nu)>0$. 
So the ratio $\alpha_{1}(\nu)/\alpha_{2}(\nu)$ defines a continuous function on the compact space $\mathbb{P}Curr(F_{N})$. Thus there exist a constant $K$ such that $\alpha_{1}(\nu)/\alpha_{2}(\nu)<K$ for all $\nu\in Curr(F_{N})-\{0\}$. Pick $M\ge1$ such that $\lambda^{M}\ge K$. Then we have 
\[
\max\{\alpha_{1}(\varphi^{m}\psi\nu),\alpha_{1}(\psi^{-1}\varphi^{-m}\nu)\}>\alpha_{1}(\nu)
\]  for all $m\ge M$ and for all $\nu\in Curr(F_{N})-\{0\}$. 
\begin{claim*} The action of $\varphi^{m}\psi$ on $Curr(F_{N})-\{0\}$ does not have periodic orbits. 
\end{claim*}
\noindent Let us set $\theta=\varphi^{m}\psi$. Assume that there exist a $\nu\in Curr(F_{N})-\{0\}$ such that $\theta^{k}(\nu)=\nu$ for some $k\ge1$. Since $\max\{\alpha_{1}(\theta\nu),\alpha_{1}(\theta^{-1}\nu)\}>\alpha_{1}(\nu)$ there are two cases to consider. If $\alpha_{1}(\theta\nu)>\alpha_{1}(\nu)$ then by induction it is straightforward to show that $\alpha_{1}(\theta^{n}\nu)>\alpha_{1}(\theta^{n-1}\nu)>\dotsc>\alpha_{1}(\nu)$ for all $n\ge1$. Similarly, $\alpha_{1}(\theta^{-1}\nu)>\alpha_{1}(\nu)$ implies that $\alpha_{1}(\theta^{-n}\nu)>\alpha_{1}(\theta^{-(n-1)}\nu)>\dotsc>\alpha_{1}(\nu)$ for all $n\ge1$. 
In any case, it is clear that $\theta^{k}(\nu)\ne\nu$ for all $k\ge1$, which is a contradiction. 

Now, observe that if $\theta=\varphi^{m}\psi$ had a periodic conjugacy class that would mean that  $\theta$ acts on $Curr(F_{N})-\{0\}$ with a periodic orbit. So $\varphi^{m}\psi$ does not have a periodic conjugacy class and hence it is hyperbolic. 

\end{proof}
\begin{prop}\cite{CP} \label{claypettet} Let $\varphi\in Out(F_{N})$ be a fully irreducible outer automorphism. Let $[T_{+}]$ and $[T_{-}]$ be the corresponding attracting and repelling trees in the closure of the projectivized Outer Space $\overline{CV}_N$. Assume $\psi\in Out(F_{N})$ is such that $[T_{+}\psi]\ne[T_{-}]$. Then there is an $M\ge 0$ such that for $m\ge M$ the element $\varphi^{m}\psi\in Out(F_{N})$ is fully irreducible. 
\end{prop}
\begin{rem} \label{mainremark} Note that $\psi H_{1}\cap H_{0}=\emptyset$ in fact implies that $[T_{+}\psi]\ne[T_{-}]$. Assume otherwise and look at the intersection number \[
0=\left<T_{+}\psi,a\mu_{+}+b\mu_\alpha\right>=\left<T_{+},\psi(a\mu_{+}+b\mu_\alpha)\right>\ne0 
\]
by Theorem \ref{Erg}, which is a contradiction. Let $\varphi\in Out(F_{N})$ be a non-atoroidal iwip, and $\psi\in Out(F_{N})$ be an element such that $\psi H_{1}\cap H_{0}=\emptyset$. Now let M be the largest of the two in previous the two lemmas, then for all $m\ge M$ the element $\varphi^{m}\psi$ is a hyperbolic iwip. 
\end{rem}

\begin{thm}\label{mainapp} Let $H\leq Out(F_N)$ such that $H$ contains an iwip $\varphi$. Then one of the following holds:
\begin{enumerate} 
\item $H$ contains a hyperbolic iwip. 
\item $H$ is geometric, i.e. $H$ contains no hyperbolic iwips and $H\leq Mod^{\pm}(S)\leq Out(F_N)$ where $S$ is a compact surface with one boundary component with $\pi_1(S)=F_N$ such that $\varphi\in H$ is induced by a pseudo-Anosov homeomorphism of $S$.  
\end{enumerate}
\end{thm}
\begin{proof} If the iwip $\varphi$ is hyperbolic, then (1) holds and there is nothing to prove. Suppose now that the iwip $\varphi$ is non-atoroidal. Then $\varphi$ is induced by a \emph{pseudo-Anosov} homeomorphism on a surface $S$ with one boundary component  $\alpha$ and $\pi_1(S)\cong F_N$. Note that if $[g]$ is the conjugacy class in $F_N$ corresponding to the boundary curve $\alpha$ of $S$, then $\varphi$ fixes $[g]$ up to a possible inversion. In this situation, the extended mapping class group $Mod^{\pm}(S)$ is naturally included as a subgroup of $Out(F_N)$. Moreover, by the Dehn-Nielsen-Baer theorem \cite{Nielsen}, the subgroup of $Out(F_{N})$, consisting of all elements of $Out(F_N)$ which fix $[g]$ up to inversion is exactly $Mod^{\pm}(S)$. If $H\le Mod^{\pm}(S)$, then part (2) of Theorem holds and there is nothing to prove. Assume now that $H$ is not contained in $Mod^{\pm}(S)$. Then there exist an element $\psi\in H$ such that $\psi([g])\ne[g^{\pm1}]$.

\begin{claim*} Let $\psi\in Out(F_{N})$ be an element such that $\psi([g])\ne[g^{\pm1}]$. Then,
\[ \psi(H_{1})\cap H_{0}=\emptyset. 
\]
\end{claim*}
\noindent First,
$
\psi[t_1\mu_{\alpha}+t_2\mu_{+}]=[t'_{1}\mu_{\alpha}+t'_{2}\mu_{-}]
$
implies $\psi[\mu_{+}]=[\mu_{-}]$. Indeed, only periodic leaves in the support of the right hand side are leaves labeled by powers of $g$. Therefore, $t_1=0$. Similarly, it is easy to see that $t'_1=0$. From here we note that $[T_{+}\psi]=[T_{-}]$. To see this, look at the intersection number
\[
0\neq\left<T_{+},\mu_{-}\right>=\left<T_{+},\psi\mu_{+}\right>=\left<T_{+}\psi,\mu_{+}\right>,
\] which implies that $T_{+}\psi=cT_{-}$ for some $c>0$ by Theorem \ref{Erg2}.  Therefore we get \[
0=c\|g\|_{T_{-}}=\|g\|_{T_{+}\psi}=\|\psi(g)\|_{T_{+}},
\]
which implies $\psi([g])=[g^{\pm1}]$ since only conjugacy classes that have translation length $0$ are peripheral curves. This contradicts the choice of $\psi$. Hence $\psi(H_{1})\cap H_{0}=\emptyset$, and the Claim is verified. \\
Since $\psi(H_1)\cap H_0=\emptyset$, Remark \ref{mainremark} implies that  $[T_{+}\psi]\ne[T_{-}]$. Therefore, by Lemma \ref{cpsimilar} and Proposition \ref{claypettet}, there exists $m\ge 1$ such that the element $\phi^m\psi\in H$ is a hyperbolic iwip.
\end{proof}

\begin{rem}\label{ranktwo} Now assume that the original $\varphi\in Out(F_N)$ in the statement of Theorem \ref{mainapp} is non-atoroidal and $(1)$ holds, i.e. there is an element $\theta\in H$ which is hyperbolic. Then $H$ is not virtually cyclic, since otherwise some positive power of a non-atoroidal iwip $\varphi$ would be equal to some positive power of a hyperbolic iwip $\theta$. Therefore, $H$ is a subgroup of $Out(F_N)$ which is not virtually cyclic and contains a hyperbolic iwip $\theta$. Therefore by Corollary 6.3 of \cite{KL5}, $H$ contains a free subgroup $L$ of rank $2$ such that all nontrivial elements of $L$ are hyperbolic iwips. 
\end{rem}

\begin{cor}
Let $N\geq3$ and $H\leq Out(F_N)$ be a nontrivial normal
subgroup. Then $H$ contains a hyperbolic iwip.
\end{cor}
\begin{proof} By Lemma 5.1 of \cite{CFKM}, the subgroup $H$ contains an iwip $\varphi$.
If $\varphi$ is hyperbolic, we are done. Assume that $\varphi$ is non-atoroidal and hence induced by a pseudo-Anosov map on a compact surface $S$ with one boundary component $\alpha$. 
\begin{claim*} If $H_0\le Mod^\pm(S)$ contains an iwip $\varphi$, then $H_0$ is not normal in $Out(F_N)$.
\end{claim*}
Suppose, on the contrary, that $H_0$ is normal in $Out(F_N)$. Choose a hyperbolic element $\eta \in Out(F_N)$ (such $\eta$ exists since $N\ge3$). Put $[g_2]=\eta [g_1]$, where $[g_1]$ represents the boundary curve of $S$. Put $\varphi_1=\eta \varphi \eta^{-1}$. Since $H_0$ is normal, then $\varphi_1\in H_0$.
Since $\eta$ has no periodic conjugacy classes, we have $[g_2]\ne [g_1^{\pm 1}]$. Then $\varphi_1[g_2]=\eta \varphi \eta^{-1} [g_2] = [g_2]$.
Since $\varphi$ is a non-atoroidal iwip, the element $\varphi_1=\eta \varphi \eta^{-1}$ is also a non-atoroidal iwip, and $[g_2]$ is a periodic conjugacy class for $\varphi_1$. A geometric iwip has a unique, up to inversion, nontrivial periodic conjugacy class $[g]$ such that $g\in F_N$ is not a proper power. Since $\varphi_1[g_2]=[g_2]$, $[g_2]\ne [g_1^{\pm 1}]$ and $g_2$ is not a proper power, it follows that $\varphi_1[g_1]\ne [g_1^{\pm 1}]$. Hence $\varphi_1\not\in Mod^\pm(S)$, contrary to the assumption that $H_0\le Mod^\pm(S)$. This verifies the Claim.

Since $H$ is normal in $Out(F_N)$ and contains a non-atoroidal iwip $\varphi$ coming from a pseudo-Anosov element of $Mod^{\pm}(S)$, the Claim implies that $H$ is not contained in $Mod^\pm(S)$.
Therefore by Theorem \ref{mainapp}, $H$ contains a hyperbolic iwip, as required.
\end{proof}

\bibliographystyle{abbrv}
\bibliography{dynamicsoncurrents}

\begin{thebibliography}{10}

\bibitem{BF}
M.~Bestvina and M.~Feighn.
\newblock A combination theorem for negatively curved groups.
\newblock {\em J. Differential Geom.}, 35(1):85--101, 1992.

\bibitem{BF93}
M.~Bestvina and M.~Feighn.
\newblock Outer limits.
\newblock 1993.
\newblock http://andromeda.rutgers.edu/~feighn/papers/outer.pdf.

\bibitem{BF10}
M.~Bestvina and M.~Feighn.
\newblock A hyperbolic {${\rm Out}(F_n)$}-complex.
\newblock {\em Groups Geom. Dyn.}, 4(1):31--58, 2010.

\bibitem{BH92}
M.~Bestvina and M.~Handel.
\newblock Train tracks and automorphisms of free groups.
\newblock {\em Ann. of Math. (2)}, 135(1):1--51, 1992.

\bibitem{Bo86}
F.~Bonahon.
\newblock Bouts des vari\'et\'es hyperboliques de dimension {$3$}.
\newblock {\em Ann. of Math. (2)}, 124(1):71--158, 1986.

\bibitem{Bo88}
F.~Bonahon.
\newblock The geometry of {T}eichm\"uller space via geodesic currents.
\newblock {\em Invent. Math.}, 92(1):139--162, 1988.

\bibitem{Brink}
P.~Brinkmann.
\newblock Hyperbolic automorphisms of free groups.
\newblock {\em Geom. Funct. Anal.}, 10(5):1071--1089, 2000.

\bibitem{CFKM}
M.~Carette, S.~Francaviglia, I.~Kapovich, and A.~Martino.
\newblock Spectral rigidity of automorphic orbits in free groups.
\newblock {\em Algebr. Geom. Topol.}, 12(3):1457--1486, 2012.

\bibitem{CFKM2}
M.~Carette, S.~Francaviglia, I.~Kapovich, and A.~Martino.
\newblock Corrigendum to "{S}pectral rigidity of automorphic orbits in free
  groups".
\newblock {\em Alg. Geom. Topol., to appear}, 2014.
\newblock http://www.math.uiuc.edu/\~{}kapovich/PAPERS/corr3.pdf.

\bibitem{CB}
A.~J. Casson and S.~A. Bleiler.
\newblock {\em Automorphisms of surfaces after {N}ielsen and {T}hurston},
  volume~9 of {\em London Mathematical Society Student Texts}.
\newblock Cambridge University Press, Cambridge, 1988.

\bibitem{CP}
M.~Clay and A.~Pettet.
\newblock Current twisting and nonsingular matrices.
\newblock {\em Comment. Math. Helv.}, 87(2):385--407, 2012.

\bibitem{CM}
M.~Culler and J.~W. Morgan.
\newblock Group actions on {${\bf R}$}-trees.
\newblock {\em Proc. London Math. Soc. (3)}, 55(3):571--604, 1987.

\bibitem{CV}
M.~Culler and K.~Vogtmann.
\newblock Moduli of graphs and automorphisms of free groups.
\newblock {\em Invent. Math.}, 84(1):91--119, 1986.

\bibitem{DLR}
M.~Duchin, C.~J. Leininger, and K.~Rafi.
\newblock Length spectra and degeneration of flat metrics.
\newblock {\em Invent. Math.}, 182(2):231--277, 2010.

\bibitem{HM}
M.~Handel and L.~Mosher.
\newblock Subgroup classification in {${\rm Out}(F_N)$}.
\newblock 2009.
\newblock preprint arXiv:0908.1255.

\bibitem{Iva}
N.~V. Ivanov.
\newblock {\em Subgroups of {T}eichm\"uller modular groups}, volume 115 of {\em
  Translations of Mathematical Monographs}.
\newblock American Mathematical Society, Providence, RI, 1992.
\newblock Translated from the Russian by E. J. F. Primrose and revised by the
  author.

\bibitem{Ka1}
I.~Kapovich.
\newblock The frequency space of a free group.
\newblock {\em Internat. J. Algebra Comput.}, 15(5-6):939--969, 2005.

\bibitem{Ka2}
I.~Kapovich.
\newblock Currents on free groups.
\newblock In {\em Topological and asymptotic aspects of group theory}, volume
  394 of {\em Contemp. Math.}, pages 149--176. Amer. Math. Soc., Providence,
  RI, 2006.

\bibitem{KL1}
I.~Kapovich and M.~Lustig.
\newblock The actions of {${\rm Out}(F_k)$} on the boundary of outer space and
  on the space of currents: minimal sets and equivariant incompatibility.
\newblock {\em Ergodic Theory Dynam. Systems}, 27(3):827--847, 2007.

\bibitem{KL2}
I.~Kapovich and M.~Lustig.
\newblock Geometric intersection number and analogues of the curve complex for
  free groups.
\newblock {\em Geom. Topol.}, 13(3):1805--1833, 2009.

\bibitem{KL3}
I.~Kapovich and M.~Lustig.
\newblock Intersection form, laminations and currents on free groups.
\newblock {\em Geom. Funct. Anal.}, 19(5):1426--1467, 2010.

\bibitem{KL5}
I.~Kapovich and M.~Lustig.
\newblock Ping-pong and outer space.
\newblock {\em J. Topol. Anal.}, 2(2):173--201, 2010.

\bibitem{Kap}
M.~Kapovich.
\newblock {\em Hyperbolic manifolds and discrete groups}, volume 183 of {\em
  Progress in Mathematics}.
\newblock Birkh\"auser Boston Inc., Boston, MA, 2001.

\bibitem{Lein}
C.~J. Leininger.
\newblock Degenerations of hyperbolic structures on surfaces.
\newblock In {\em Geometry, topology and dynamics of character varieties},
  volume~23 of {\em Lect. Notes Ser. Inst. Math. Sci. Natl. Univ. Singap.},
  pages 95--138. World Sci. Publ., Hackensack, NJ, 2012.

\bibitem{LL}
G.~Levitt and M.~Lustig.
\newblock Irreducible automorphisms of {$F_n$} have north-south dynamics on
  compactified outer space.
\newblock {\em J. Inst. Math. Jussieu}, 2(1):59--72, 2003.

\bibitem{Martin}
R.~Martin.
\newblock {\em Non-uniquely ergodic foliations of thin-type, measured currents
  and automorphisms of free groups}.
\newblock ProQuest LLC, Ann Arbor, MI, 1995.
\newblock Thesis (Ph.D.)--University of California, Los Angeles.

\bibitem{Nielsen}
J.~Nielsen.
\newblock Untersuchungen zur {T}opologie der geschlossenen zweiseitigen
  {F}l\"achen.
\newblock {\em Acta Math.}, 50(1):189--358, 1927.

\bibitem{Otal}
J.-P. Otal.
\newblock Le spectre marqu\'e des longueurs des surfaces \`a courbure
  n\'egative.
\newblock {\em Ann. of Math. (2)}, 131(1):151--162, 1990.

\bibitem{Th}
W.~P. Thurston.
\newblock On the geometry and dynamics of diffeomorphisms of surfaces.
\newblock {\em Bull. Amer. Math. Soc. (N.S.)}, 19(2):417--431, 1988.

\bibitem{U}
C.~Uyanik.
\newblock Dynamics of hyperbolic iwips.
\newblock {\em arXiv preprint arXiv:1401.1557}, 2014.

\bibitem{WH}
J.~H. Whitehead.
\newblock On certain sets of elements in a free group.
\newblock {\em Proceedings of the London Mathematical Society}, 2(1):48--56,
  1936.

\end{thebibliography}

\end{document}